\documentclass{elsarticle}

\usepackage{booktabs}
\usepackage{amsmath}
\usepackage{amsthm}
\usepackage{amsfonts}
\usepackage{amssymb}
\usepackage{url}
\usepackage{enumerate}
\usepackage{csquotes}

\newcommand{\F}{{\mathbb F}}
\newcommand{\Z}{{\mathbb Z}}
\newcommand{\Fix}{{\mathcal F}}
\newcommand{\D}{{\mathcal D}}
\newcommand{\gauss}[3]{\genfrac{[}{]}{0pt}{}{#1}{#2}_{#3}}
\newcommand{\nequiv}{\not\equiv}
\newcommand{\kmmat}{M}
\DeclareMathOperator{\Aut}{Aut}
\DeclareMathOperator{\GammaL}{\Gamma L}
\DeclareMathOperator{\PGammaL}{P\Gamma L}
\DeclareMathOperator{\PGL}{PGL}
\DeclareMathOperator{\GL}{GL}
\DeclareMathOperator{\PG}{PG}
\DeclareMathOperator{\id}{id}

\newtheorem{theorem}{Theorem}
\newtheorem{corollary}[theorem]{Corollary}
\newtheorem{lemma}{Lemma}[section]

\newtheorem{definition}[lemma]{Definition}
\newtheorem{remark}[lemma]{Remark}
\newtheorem{example}[lemma]{Example}

\makeatletter
\def\ps@pprintTitle{%
\let\@oddhead\@empty
\let\@evenhead\@empty
\def\@oddfoot{\centerline{\thepage}}%
\let\@evenfoot\@oddfoot}
\makeatother

\begin{document}
\begin{frontmatter}

\title{On the automorphism group\\of a binary $q$-analog of the Fano plane}

\author{Michael Braun}
\address{Faculty of Computer Science, University of Applied Sciences, Darmstadt, Germany}

\author{Michael Kiermaier}
\address{Lehrstuhl II f\"ur Mathematik, Universit\"at Bayreuth, Germany}

\author{Anamari Naki\'{c}}
\address{Faculty of Electrical Engineering and Computing, University of Zagreb, Croatia}

\begin{abstract}
The smallest set of admissible parameters of a $q$-analog of a Steiner system is $S_2[2,3,7]$.
The existence of such a Steiner system -- known as a binary $q$-analog of the Fano plane -- is still open.
In this article, the automorphism group of a putative binary $q$-analog of the Fano plane is investigated by a combination of theoretical and computational methods.
As a conclusion, it is either rigid or its automorphism group is cyclic of order $2$, $3$ or $4$.
Up to conjugacy in $\GL(7,2)$, there remains a single possible group of order $2$ and $4$, respectively, and two possible groups of order $3$.
For the automorphisms of order $2$, we give a more general result which is valid for any binary $q$-Steiner triple system.
\end{abstract}

\end{frontmatter}

\section{Introduction}
\label{sect:introduction}
Due to the application in error-correction in randomized network coding \cite{KK08}, $q$-analogs of combinatorial designs have gained a lot of interest lately.
Arguably the most important open problem in this field is the question of the existence of a $2$-$(7,3,1)_q$ design, as it has the smallest admissible parameter set of a non-trivial $q$-Steiner system with $t \geq 2$.
It is known as a $q$-analog of the Fano plane and has been tackled in many articles~\cite{EV11,HS11,KP14,Met99,MMY95,Tho87,Tho96}.

In this paper we investigate the automorphism group of a putative binary $q$-analog of the Fano plane.
The following result will be proven in Section~\ref{sect:autfano}.
\begin{theorem}\label{tm:main}
A binary $q$-analog of the Fano plane is either rigid or its automorphism group is cyclic of order $2$, $3$ or $4$.
Representing the automorphism group as a subgroup of $\GL(7,2)$, up to conjugacy it is contained in the following list:
\begin{enumerate}[(a)]
\item The trivial group\text{.}
\item The group of order $2$
\[
	\left\langle\left(
	\begin{smallmatrix}
	0& 1& 0& 0& 0& 0& 0 \\
	1& 0& 0& 0& 0& 0& 0 \\
	0& 0& 0& 1& 0& 0& 0 \\
	0& 0& 1& 0& 0& 0& 0 \\
	0& 0& 0& 0& 0& 1& 0 \\
	0& 0& 0& 0& 1& 0& 0 \\
	0& 0& 0& 0& 0& 0& 1 
	\end{smallmatrix}\right)\right\rangle\text{.}
\]
\item One of the following two groups of order $3$:
\[
\left\langle
\left(
\begin{smallmatrix}
0& 1& 0& 0& 0& 0& 0\\ 
1& 1& 0& 0& 0& 0& 0\\ 
0& 0& 0& 1& 0& 0& 0\\
0& 0& 1& 1& 0& 0& 0\\
0& 0& 0& 0& 0& 1& 0\\
0& 0& 0& 0& 1& 1& 0\\
0& 0& 0& 0& 0& 0& 1
\end{smallmatrix}
\right)
\right\rangle
\quad
\text{and}
\quad
\left\langle
\left(
\begin{smallmatrix}
0& 1& 0& 0& 0& 0& 0\\ 
1& 1& 0& 0& 0& 0& 0\\ 
0& 0& 0& 1& 0& 0& 0\\ 
0& 0& 1& 1& 0& 0& 0\\ 
0& 0& 0& 0& 1& 0& 0\\ 
0& 0& 0& 0& 0& 1& 0\\ 
0& 0& 0& 0& 0& 0& 1
\end{smallmatrix}
\right)\right\rangle\text{.}
\]
\item The cyclic group of order $4$
\[
\left\langle
\left(
\begin{smallmatrix}
1&1&0&0&0&0&0 \\
0&1&1&0&0&0&0 \\
0&0&1&0&0&0&0 \\
0&0&0&1&1&0&0 \\
0&0&0&0&1&1&0 \\
0&0&0&0&0&1&1 \\
0&0&0&0&0&0&1
\end{smallmatrix}\right)
\right\rangle\text{.}
\]
\end{enumerate}
\end{theorem}

The idea of eliminating automorphism groups has been
used in the existence problems for other notorious combinatorial objects. Probably the most prominent example is the projective plane of order $10$. In \cite{W79a,W79b} it was shown that the order of an automorphism is either $3$ or $5$. Those two orders had been eliminated later in \cite{AHT80,JT81}, implying that the automorphism group of a projective plane of order $10$ must be trivial.
Finally, the non-existence has been shown in \cite{LTS89} in an extensive computer search. 

There are other examples of existence problems where the same idea has been used, that remain still open.
In coding theory, the question of the existence of a binary doubly-even self-dual code with the parameters $[72,36,16]$ was raised more than 40 years ago in \cite{S73}.
Its automorphisms have been heavily investigated in a series of articles.
After the latest result \cite{YY14}, it is known that its automorphism group is either trivial, cyclic of order $2$, $3$ or $5$, or a Klein four group.
In classical design theory, the smallest $v$ for which the existence of a $3$-design of order $v$ is undecided is $16$; indeed a $3$-$(16, 7, 5)$ design is still unknown.
Recent results show that if such design exists, then its automorphism group is either trivial or a $2$-group \cite{E10, N15}.

The proof of Theorem~\ref{tm:main} is partially based on the following more general result on automorphisms of order $2$ of binary $q$-Steiner triple systems, which will be shown in Section~\ref{sect:triple}.
\begin{theorem}\label{thm:aut2}
Let $\D$ be a binary $S_2[2,3,v]$ $q$-Steiner triple system and $A\in\GL(v,2)$ the matrix representation of an automorphism of $\D$ of order $2$.
For $s\in\{1,\ldots,\lfloor v/2\rfloor\}$ let $A_{v,s}$ denote the $v\times v$ block diagonal matrix built from $s$ blocks $\left(\begin{smallmatrix}0 & 1\\1 & 0\end{smallmatrix}\right)$ followed by a $(v-2s)\times (v-2s)$ unit matrix.
\begin{enumerate}[(a)]
\item In the case $v\equiv 1\bmod {6}$, $A$ is conjugate to a matrix $A_{v,s}$ with $3\mid s$.
\item In the case $v\equiv 3\bmod {6}$, $A$ is conjugate to a matrix $A_{v,s}$ with $s\nequiv 2\bmod {3}$.
\end{enumerate}
\end{theorem} 

\section{Preliminaries}
\subsection{Group actions}
Suppose that a finite group $G$ is acting on a finite set $X$.
For $x \in X$ and $\alpha \in G$, the image of $x$ under $\alpha$ is denoted by $x^{\alpha}$.
The set $x^G = \{x^{\alpha} \mathrel{:} \alpha \in G\}$ is called the \emph{orbit of $x$ under $G$} or the \emph{$G$-orbit of $x$}, in short.
The set of all orbits is a partition of $X$.
We say that $x \in X$ is \emph{fixed} by $G$ if $x^{G} = \{x\}$.
So the fixed elements correspond to the orbits of length $1$.
The set $G_x = \{\alpha \in G \,:\, x^{\alpha} = x \}$ is a subgroup of $G$, called the \emph{stabilizer} of $x$ in $G$.
For all $g\in G$ we have
\begin{equation}
	\label{eq:stab_conj}
	G_{x^{g}} = g^{-1} G_x g\text{.}
\end{equation}
In particular the stabilizers of elements in the same orbit are conjugate, and any conjugate subgroup $H$ of $G_x$ is the stabilizer of some element in the orbit of $x$.
The Orbit-Stabilizer Theorem states that $\#x^G = [G : G_x]$,%
\footnote{The symbol $\#$ denotes the cardinality of a set.}
implying that $\#x^G$ divides $\#G$.
The action of $G$ on $X$ extends in the obvious way to an action on the set of subsets of $X$:
For $S\subseteq X$, we set $S^\alpha = \{x^{\alpha} \,:\, x \in S\}$.
For further reading on the theory of finite group actions, see \cite{Ker99}.

\subsection{Linear and semilinear maps}
Let $V,W$ be vector spaces over a field $F$.
A map $f : V \to W$ is called \emph{semilinear} if it is additive and if there is a field automorphism $\sigma\in\Aut(F)$ such that $f(\lambda\mathbf{v}) = \sigma(\lambda) f(\mathbf{v})$ for all $\lambda\in F^\times$ and all $\mathbf{v}\in V$.
The map $f$ is \emph{linear} if and only if $\sigma = \id_F$.
In fact, the \emph{general linear group} $\GL(V)$ of all linear bijections of $V$ is a normal subgroup of the \emph{general semilinear group} $\GammaL(V)$ of all semilinear bijections of $V$.
Moreover, $\GammaL(V)$ decomposes as a semidirect product $\GL(V) \rtimes \Aut(F)$.
In this representation, the natural action on $V$ takes the form $(\phi,\sigma)(\mathbf{v}) = \phi(\sigma(\mathbf{v}))$, where $\sigma(\mathbf{v})$ is the component-wise application of $\sigma$ to the coordinate vector of $\mathbf{v}$ with respect to some fixed basis.

The center $Z(\GL(V))$ of $\GL(V)$ consists of all diagonal maps $\mathbf{v} \mapsto \lambda\mathbf{v}$ with $\lambda\in F^\times$.
It is a normal subgroup of $\GL(V)$ and of $\GammaL(V)$.
The quotient group $\GL(V)/Z(\GL(V))$ is known as the \emph{projective linear group} $\PGL(V)$, and the quotient group $\GammaL(V)/Z(\GL(V))$ is known as the \emph{projective semilinear group} $\PGammaL(V)$.

\subsection{The subspace lattice}
Let $V$ denote a $v$-dimensional vector space over the finite field $\F_q$ with $q$ elements.
For simplicity, a subspace of $V$ of dimension $r$ will be called an \emph{$r$-subspace}.
The set of all $r$-subspaces of $V$ is called the \emph{Gra{\ss}mannian} and is denoted by $\gauss{V}{r}{q}$.
As in projective geometry, the $1$-subspaces of $V$ are called \emph{points}, the $2$-subspaces \emph{lines} and the $3$-subspaces \emph{planes}.
Our focus lies on the case $q = 2$, where the $1$-subspaces $\langle\mathbf{x}\rangle_{\F_2}\in \gauss{V}{1}{2}$ are in one-to-one correspondence with the nonzero vectors $\mathbf{x} \in V\setminus\{\mathbf{0}\}$.
The number of all $r$-subspaces of $V$ is given by
$$
\#\gauss{V}{r}{q} = \gauss{v}{r}{q} = \frac{(q^v-1)\cdots(q^{v-r+1}-1)}{(q^r-1)\cdots(q-1)}.
$$
The set $\mathcal{L}(V)$ of all subspaces of $V$ forms the subspace lattice of $V$.

By the fundamental theorem of projective geometry, for $v\neq 2$ the automorphism group of $\mathcal{L}(V)$ is given by the natural action of $\PGammaL(V)$ on $\mathcal{L}(V)$.
In the case that $q$ is prime, the group $\PGammaL(V)$ reduces to $\PGL(V)$, and for the case of our interest $q = 2$, it reduces further to $\GL(V)$.
After a choice of a basis of $V$, its elements are represented by the invertible $v\times v$ matrices $A$, and the action on $\mathcal{L}(V)$ is given by the vector-matrix-multiplication $\mathbf{v} \mapsto \mathbf{v} A$.

\subsection{Designs}
\begin{definition}
Let $t,v,k$ be integers with $0 \leq t \leq k\leq v$, $\lambda$ another positive integer and $\mathcal{P}$ a set of size $v$.
A $t$-$(v,k,\lambda)$ \emph{(combinatorial) design} $\D$ is a set of $k$-subsets of $\mathcal{P}$ (called \emph{blocks}) such that every $t$-subset of $\mathcal{P}$ is contained in exactly $\lambda$ blocks of $\D$.
When $\lambda = 1$, $\D$ is called a \emph{Steiner system} and denoted by $S(t,k,v)$. 
If additionally $t=2$ and $k=3$, $\D$ is called a \emph{Steiner triple system}.
\end{definition}

For the definition of a $q$-analog of a combinatorial design, the subset lattice on $\mathcal{P}$ is replaced by the subspace lattice of $V$.
In particular, subsets are replaced by subspaces and cardinality is replaced by dimension.
So we get \cite{Cam74,Cam74:a}:

\begin{definition}
Let $t,v,k$ be integers with $0 \leq t \leq k\leq v$, $\lambda$ another positive integer and $V$ an $\F_q$-vector space of dimension $v$.
A $t$-$(v,k,\lambda)_q$ \emph{design} $\D$ over the field $\F_q$ is a set of $k$-subspaces of $V$ (called \emph{blocks}) such that every $t$-subspace of $V$ is contained in exactly $\lambda$ blocks of $\D$.
When $\lambda = 1$, $\D$ is called a \emph{$q$-Steiner system} and denoted by $S_q[t,k,v]$. 
If additionally $t=2$ and $k=3$, $\D$ is called a \emph{$q$-Steiner triple system}.
\end{definition}

There are good reasons to consider the subset lattice as a subspace lattice over the unary \enquote{field} $\F_1$ \cite{Cohn-2004}.
Thus, classical combinatorial designs can be seen as the limit case $q=1$ of a design over a finite field.
Indeed, quite a few statements about combinatorial designs have a generalization to designs over finite fields, such that the case $q = 1$ reproduces the original statement \cite{BKKL,KL,KP14,NP14}.

One example of such a statement is the following \cite[Lemma~4.1(1)]{Suzuki-1990}:
If $\D$ is a $t$-$(v, k, \lambda_t)_q$ design, then $\D$ is also an $s$-$(v,k,\lambda_s)_q$ for all $s\in\{0,\ldots,t\}$, where
\[
\lambda_s := \lambda_t \frac{\gauss{v-s}{t-s}{q}}{\gauss{k-s}{t-s}{q}}.
\]
In particular, the number of blocks in $\D$ equals
\[
\#\D = \lambda_0 = \lambda_t \frac{\gauss{v}{t}{q}}{\gauss{k}{t}{q}}.
\]
So a necessary condition on the existence of a design with parameters $t$-$(v, k, \lambda)_q$ is that for all $s\in\{0,\ldots,t\}$ the numbers $\lambda \gauss{v-s}{t-s}{q}/\gauss{k-s}{t-s}{q}$ are integers (\emph{integrality conditions}).
In this case, the parameter set $t$-$(v,k,\lambda)_q$ is called \emph{admissible}.
It is further called \emph{realizable} if a $t$-$(v,k,\lambda)_q$ design actually exists.

For designs over finite fields, the action of $\Aut(\mathcal{L}(V)) \cong \PGammaL(V)$ on $\mathcal{L}(V)$ provides a notion of isomorphism.
Two designs in the same ambient space $V$ are called \emph{isomorphic} if they are contained in the same orbit of this action (extended to the power set of $\mathcal{L}(V)$).
The \emph{automorphism group} $\Aut(\D)$ of a design $\D$ is its stabilizer with respect to this group action.
If $\Aut(\D)$ is trivial, we will call $\D$ \emph{rigid}.
Furthermore, for $G \leq \PGammaL(V)$, $\D$ will be called $G$-invariant if it is fixed by all $g\in G$ or equivalently, if $G\leq \Aut(\D)$.
Note that if $\D$ is $G$-invariant, then $\D$ is also $H$-invariant for all subgroups $H \leq G$. 

\subsection{Admissibility and realizability}
The question of the realizability of an admissible parameter set is very hard to answer in general.
In the special case $t = 1$, $q$-Steiner systems are called \emph{spreads}.
It is known that the spread parameters $S_q[1,k,v]$ are realizable if and only if they are admissible if and only if $k$ divides $v$.
However for $t\geq 2$, the problem tantalized many researchers \cite{Beu78, BEO+, EV11,Met99,Tho87,Tho96}.
Only recently, the first example of such a $q$-Steiner system was constructed \cite{BEO+}.
It is a $q$-Steiner triple system with the parameters $S_2[2,3,13]$.

As a direct consequence of the integrality conditions, the parameter set of a Steiner triple system $S(2,3,v)$ or a $q$-Steiner triple system $S_q[2,3,v]$ is admissible if and only if $v\equiv 1,3\bmod 6$ and $v\geq 7$.
In the ordinary case $q=1$, the existence question is completely answered by the result that a Steiner triple system is realizable if and only if it is admissible \cite{Kirkman}.
However in the $q$-analog case, our current knowledge is quite sparse.
The only decided case is given by the above mentioned existence of an $S_2[2,3,13]$.

The smallest admissible case of a $q$-Steiner triple system is $S_q[2,3,7]$, whose existence is open for any prime power value of $q$.
It is known as a \emph{$q$-analog of the Fano plane}, since the unique Steiner triple system $S(2,3,7)$ is the Fano plane.
It is worth noting that there are cases of Steiner systems without a $q$-analog, as the famous large Witt design with parameters $S(5,8,24)$ does not have a $q$-analog for any prime power $q$ \cite{KL}. 

\subsection{The method of Kramer and Mesner}
The method of Kramer and Mesner \cite{KM76} is a powerful tool for the computational construction of combinatorial designs.
It has been successfully adopted and used for the construction of designs over a finite field \cite{MMY95,BKL05}.
For example, the hitherto only known $q$-analog of a Steiner triple system in \cite{BEO+} has been constructed by this method.
Here we give a short outline, for more details we refer the reader to \cite{BKL05}. 
As another computational construction method we mention the use of tactical decompositions~\cite{JT85}, which has been adopted to the $q$-analog case in~\cite{NP14}.

The \emph{Kramer-Mesner matrix} $\kmmat_{t,k}^G$ is defined to be the matrix whose rows and columns are indexed by the $G$-orbits on the set $\gauss{V}{t}{q}$ of $t$-subspaces and on the set $\gauss{V}{k}{q}$ of $k$-subspaces of $V$, respectively.
The entry of $\kmmat_{t,k}^G$ with row index $T^G$ and column index $K^G$ is defined as $\#\{K'\in K^G \mid T\subseteq K'\}$.
Now there exists a $G$-invariant $t$-$(v,k,\lambda)_q$ design if and only if there is a zero-one solution vector $\mathbf{x}$ of the linear equation system
\begin{equation}\label{eq:km}
\kmmat_{t,k}^{G}\mathbf{x}=\lambda \mathbf{1},
\end{equation} 
where $\textbf{1}$ denotes the all-one column vector. 
More precisely, if $\mathbf{x}$ is a zero-one solution vector of the system~\eqref{eq:km}, a $t$-$(v,k,\lambda)_q$ design is given by the union of all orbits $K^G$ where the corresponding entry in $\mathbf{x}$ equals one.
If $\mathbf{x}$ runs over all zero-one solutions, we get all $G$-invariant $t$-$(v,k,\lambda)_q$ designs in this way.

\subsection{Normal forms for square matrices}
Let $F$ be a field and $V$ a vector space over $F$ of finite dimension $v$.
Two elements of $\GL(V)$ are conjugate if and only if their transformation matrices are conjugate in the matrix group $\GL(v,F)$ if and only if the transformation matrices are similar in the sense of linear algebra.
If $F$ is algebraically closed, representatives of the matrices in $F^{v\times v}$ up to similarity are given by the Jordan normal forms.
In the following, we discuss two common normal forms for the more general case that $F$ is not algebraically closed, the Frobenius normal form and the generalized Jordan normal form.

For a monic polynomial
\[
    f = a_0 + a_1 X + \ldots + a_{n-1} X^{n-1} + X^n\in F[X]\text{,}
\]
the \emph{companion matrix} $A_f$ of $f$ is the $n\times n$ matrix over $F$ defined as
\[
	A_f
	= \begin{pmatrix}
	    0 & 1 & 0 & \cdots & 0 \\
	    \vdots & \ddots & \ddots & \ddots & \vdots \\
	    \vdots & & \ddots & \ddots & 0 \\
	    0 & \cdots & \cdots & 0 & 1 \\
	    -a_0 & -a_1 & -a_2 & \cdots & -a_{n-1}
	\end{pmatrix}\text{.}
\]

It is known that for any square matrix $A$ over $F$ there are unique monic \emph{invariant factors} $f_1 \mid \ldots \mid f_s \in F[X]$ such that $A$ is conjugate to a block diagonal matrix consisting of the blocks $A_{f_1},\ldots,A_{f_s}$.
This matrix is called the \emph{Frobenius normal form} or the \emph{rational normal form} of $A$.
The last invariant factor $f_s$ equals the minimal polynomial of $A$ and the product of all the invariant factors equals the characteristic polynomial of $A$.
The Frobenius normal form corresponds to a decomposition of $V$ into a minimum number of $A$-cyclic subspaces. (A subspace $U$ is called \emph{$A$-cyclic} if there exists a vector $\mathbf{v}\in V$ such that $U = \langle\mathbf{v}, A\mathbf{v}, A^2\mathbf{v},\ldots\rangle_F$.)

By further decomposing the invariant factors into powers of irreducible polynomials, one arrives at an alternative matrix normal form:
For a positive integer $m$ and a monic irreducible polynomial $f$, we define the \emph{Jordan block}
\[
	J_{f,m}
	= \begin{pmatrix}
		A_f & U & \mathbf{0} & \cdots & \mathbf{0} \\
		\mathbf{0}   & A_f & \ddots & \ddots & \vdots \\
		\vdots & \ddots & \ddots & \ddots & \mathbf{0} \\
		\vdots & & \ddots & A_f & U \\
		\mathbf{0} & \cdots & \cdots & \mathbf{0} & A_f
	\end{pmatrix}
\]
with $m$ consecutive blocks $A_f$ on the diagonal.
Here, $\mathbf{0}$ denotes the $n\times n$ zero matrix and $U$ denotes the $n\times n$ matrix whose only nonzero entry is an entry $1$ in the lower left corner.
Each square matrix over $F$ is conjugate to a \emph{(generalized) Jordan normal form}, which is a block diagonal matrix consisting of Jordan blocks $J_{g_1,m_1},\ldots,J_{g_s,m_s}$ with monic irreducible polynomials $g_i$.
Furthermore, the Jordan normal form is unique up to permuting the Jordan blocks.
The Jordan normal form corresponds to a decomposition of $V$ into a maximal number of $A$-invariant subspaces.
For that reason, the matrix blocks of the Jordan normal form are typically smaller than those in the Frobenius normal form.
Depending on the application, this might be an advantage.

\section{Automorphisms of order $2$ of binary $q$-Steiner triple systems}
\label{sect:triple}
In this section we prove Theorem~\ref{thm:aut2}, which is a result on the automorphisms of order $2$ of a binary $q$-Steiner triple system.

\begin{lemma}
	\label{lma:order2}
    In $\GL(v,2)$ there are exactly $\lfloor v/2\rfloor$ conjugacy classes of elements of order $2$.
    Representatives are given by the block-diagonal matrices $A_{v,s}$ with $s\in\{1,\ldots,\lfloor v/2\rfloor\}$, consisting of $s$ consecutive $2\times 2$ blocks $\left(\begin{smallmatrix}0 & 1\\1 & 0\end{smallmatrix}\right)$, followed by a $(v-2s)\times (v-2s)$ unit matrix.
\end{lemma}

\begin{proof}
	Let $A\in\GL(v,2)$.
	The matrix $A$ is of order $2$ if and only if its minimal polynomial is $X^2 - 1 = (X+1)^2$.
	Equivalently, all the invariant factors of $A$ are $X+1$ or $(X+1)^2$, and the invariant factor $(X+1)^2$ must appear at least once.
	So for the number $s$ of the invariant factor $(X+1)^2$, there are the possibilities $s\in\{1,\ldots,\lfloor v/2\rfloor\}$.
	Representatives of the elements of order $2$ are now given by the associated Frobenius normal forms.
	The invariant factor $X+1$ gives $1\times 1$ blocks $(1)$, and the invariant factor $(X+1)^2 = X^2 + 1$ gives $2\times 2$ blocks $\left(\begin{smallmatrix}0 & 1\\1 & 0\end{smallmatrix}\right)$.
	By putting the $1\times 1$ blocks at the end rather than at the beginning, we attain the matrices $A_{v,s}$.
\end{proof}

For a matrix $A$ of order $2$, the unique conjugate $A_{v,s}$ given by Lemma~\ref{lma:order2} will be called the \emph{type} of $A$.
If $A$ is an automorphism of a design $\mathcal{D}$, by equation~\eqref{eq:stab_conj} there is an isomorphic design having the automorphism $A_{v,s}$.
Therefore, for the investigation of the action of $\langle A\rangle$ on $\mathcal{D}$, we may assume $A = A_{v,s}$ without loss of generality.

\begin{lemma}
	\label{lem:fixp}
	Let $A$ be a matrix of order $2$ and type $A_{v,s}$.
	The action of $\langle A\rangle$ partitions the point set $\gauss{\F_2^v}{1}{q}$ into $2^{v-s} - 1$ fixed points and $2^{v-s-1} (2^s - 1)$ orbits of length $2$.
\end{lemma}

\begin{proof}
Me may assume $A = A_{v,s}$.
By the Orbit-Stabilizer Theorem and $\#\langle A_{v,s}\rangle = 2$, all orbits are of length $1$ or $2$.
From the explicit form of the matrices $A_{v,s}$, it is clear that a vector $\mathbf{x} = (x_1,x_2,\ldots,x_v)\in\F_2^v$ is fixed by $A_{v,s}$ if and only if
\[
    x_1 = x_2\text{,}\quad x_3 = x_4\text{,}\ldots\text{,}\quad x_{2s - 1} = x_{2s}\text{.}
\]
Thus, each of these $2^{v-s} - 1$ vectors forms an orbit of length $1$, and the remaining $(2^v - 1) - (2^{v-s} - 1) = 2\times 2^{v-s-1} (2^s - 1)$ vectors are paired into $2^{v-s-1}(2^s - 1)$ orbits of length $2$.
\end{proof}

\begin{example}
	\label{ex:con2:v3}
	A model of the classical Fano plane is given by the points and the planes in $\PG(2,2) = \PG(\F_2^3)$.
	Its automorphism group is $\GL(3,2)$.
	By Lemma~\ref{lma:order2}, there is a single conjugacy class of automorphism types of order $2$, represented by the matrix
	\[
		A_{3,1} = \begin{pmatrix}
			0 & 1 & 0 \\
			1 & 0 & 0 \\
			0 & 0 & 1
		\end{pmatrix}\text{.}
	\]
	The action of $\langle A_{3,1}\rangle$ partitions the point set $\gauss{\F_2^3}{1}{2}$ into the $3$ fixed points
	\[
		\langle(0, 0, 1)\rangle_{\F_2}\text{,}\quad
		\langle(1, 1, 0)\rangle_{\F_2}\text{,}\quad
		\langle(1, 1, 1)\rangle_{\F_2}\text{,}
	\]
	and the two orbits of length $2$
	\[
		\{
		\langle(1, 0, 0)\rangle_{\F_2},\;
		\langle(0, 1, 0)\rangle_{\F_2}
		\}
		\quad\text{and}\quad
		\{
		\langle (1, 0, 1)\rangle_{\F_2},\;
		\langle (0, 1, 1)\rangle_{\F_2}
		\}
		\text{.}
	\]
\end{example}

\begin{example}
	\label{ex:con2:v7}
	By Lemma~\ref{lma:order2}, the conjugacy classes of elements of order $2$ in $\GL(7,2)$ are represented by
\begin{align*}
A_{7,1}
& = \left(\begin{smallmatrix}
    0& 1& 0& 0& 0& 0& 0 \\
    1& 0& 0& 0& 0& 0& 0 \\
    0& 0& 1& 0& 0& 0& 0 \\
    0& 0& 0& 1& 0& 0& 0 \\
    0& 0& 0& 0& 1& 0& 0 \\
    0& 0& 0& 0& 0& 1& 0 \\
    0& 0& 0& 0& 0& 0& 1 
\end{smallmatrix}\right)\text{,}
& A_{7,2}
& = \left(\begin{smallmatrix}
    0& 1& 0& 0& 0& 0& 0 \\
    1& 0& 0& 0& 0& 0& 0 \\
    0& 0& 0& 1& 0& 0& 0 \\
    0& 0& 1& 0& 0& 0& 0 \\
    0& 0& 0& 0& 1& 0& 0 \\
    0& 0& 0& 0& 0& 1& 0 \\
    0& 0& 0& 0& 0& 0& 1 
\end{smallmatrix}\right)\text{,}
& A_{7,3}
& = \left(\begin{smallmatrix}
    0& 1& 0& 0& 0& 0& 0 \\
    1& 0& 0& 0& 0& 0& 0 \\
    0& 0& 0& 1& 0& 0& 0 \\
    0& 0& 1& 0& 0& 0& 0 \\
    0& 0& 0& 0& 0& 1& 0 \\
    0& 0& 0& 0& 1& 0& 0 \\
    0& 0& 0& 0& 0& 0& 1 
\end{smallmatrix}\right)\text{.}
\end{align*}
The number of fixed points is $63$, $31$ and $15$, and the number of orbits of length~$2$ is $32$, $48$ and $56$, respectively.
\end{example}

Now we investigate the automorphisms of order $2$ of an $S_2[2,3,v]$ $q$-Steiner triple system $\D$.
For the remainder of the article, we will assume that $V = \F_2^v$.
This allows us to identify $\GL(V)$ with the matrix group $\GL(v,2)$.

\begin{lemma}\label{lem:fix_2}
Let $\D$ be an $S_2[2,3,v]$ $q$-Steiner system with an automorphism $A$ of order $2$ and type $A_{v,s}$.
Then each block of $\D$ fixed under the action of $\langle A\rangle$ contains either $3$ or $7$ fixed points.
The number of fixed blocks of the first type is
\begin{equation}\label{eq:f3}
F_3 = 2^{v-s-2}(2^s - 1)
\end{equation}
and the number of fixed blocks of the second type is
\begin{equation}\label{eq:f7}
F_7 = \frac{2^{2v-2s-1}+1-3\cdot 2^{v-s-2}(2^s+1)}{21}\text{.}
\end{equation}
\end{lemma}

\begin{proof}
Without restriction $A = A_{v,s}$.
Let $\Fix$ denote the set of blocks of $\D$ fixed by $G = \langle A_{v,s}\rangle$. 
The restriction of $A_{v,s}$ to any fixed block $K\in\Fix$ is an automorphism of $\mathcal{L}(K)\cong \mathcal{L}(\F_2^3)$ whose order divides $2$.
By Example~\ref{ex:con2:v3}, the restriction is either of the unique conjugacy type of order $2$ or the identity.
So the number of fixed points in $K$ is either $3$ or $7$.
Let $\Fix_3 \subseteq \Fix$ denote the set of all fixed blocks with $3$ fixed points and $\Fix_7$ those having $7$ fixed points.

We double count the set $X$ of all pairs $(\{P_1,P_2\},B)$ where $\{P_1,P_2\}$ is a point-orbit of length $2$ and $B$ is a fixed block of $\D$ passing through $P_1$ and $P_2$.
By Lemma~\ref{lem:fixp}, the number of choices for $\{P_1,P_2\}$ is $2^{v-s-1}(2^s - 1)$.
By the design property, there is exactly $\lambda = 1$ block $B$ of $\D$ passing through the line $L = \langle P_1,P_2\rangle_{\F_2}$.
Since $\{P_1,P_2\}$ is fixed under the action of $A_{v,s}$, so is $L$.
This implies that every block $B'$ in the orbit of $B$ passes through $L$, too.
By $\lambda = 1$, we get $B' = B$, so $B$ is a fixed block.
Thus $\#X = 2^{v-s-1}(2^s - 1)$.

Now we first count the possibilities for the fixed block $B$.
Since $B$ contains a point-orbit of length $2$, necessarily $B\in\Fix_3$.
By Example~\ref{ex:con2:v3}, each such $B$ contains exactly two point-orbits of length $2$.
So $\#X = 2F_3$, which verifies equation~\eqref{eq:f3}.

Now we double count the set $Y$ of all pairs $(\{P_1,P_2\},B)$ where $P_1,P_2$ are two distinct fixed points and $B$ is a fixed block of $\D$ passing through $P_1$ and $P_2$.
By Lemma~\ref{lem:fixp}, the number of choices for $\{P_1,P_2\}$ is $\binom{2^{v-s}-1}{2} = (2^{v-s} - 1)(2^{v-s-1} - 1)$.
As above, there is a single fixed block of $\D$ passing through $P_1$ and $P_2$, which yields $\#Y = (2^{v-s} - 1)(2^{v-s-1} - 1)$.

Vice versa, for $B\in\mathcal{F}_3$ there are $\binom{3}{2} = 3$ choices for $\{P_1,P_2\}$ and for $B\in\mathcal{F}_7$ the number of choices is $\binom{7}{2} = 21$.
So $\#Y = 3F_3 + 21 F_7$.
This shows that $(2^{v-s} - 1)(2^{v-s-1} - 1) = 3F_3 + 21 F_7$.
Replacing $F_3$ by formula~\eqref{eq:f3}, we obtain formula~\eqref{eq:f7}.
\end{proof}

\begin{corollary}
\label{cor:ord2}
Let $\D$ be an $S_2[2,3,v]$ $q$-Steiner system.
\begin{enumerate}[(a)]
	\item In the case $v \equiv 1\bmod 6$, $\D$ does not have an automorphism of type $A_{v,s}$ with $3\nmid s$.
	\item In the case $v \equiv 3\bmod 6$, $\D$ does not have an automorphism of type $A_{v,s}$ with $s\equiv 2\bmod 3$.
\end{enumerate}
\end{corollary}

\begin{proof}
By definition, the number $F_7$ of Lemma~\ref{lem:fix_2} is an integer.
So the numerator of the right hand side of \eqref{eq:f7} must be a multiple of $7$.

We compute its remainder modulo $7$.
The multiplicative order of $2$ modulo $7$ equals $3$, so the remainder only depends on the values of $v$ and $s$ modulo $3$.
We get:
\[
	\begin{array}{c|ccc}
	& s\equiv 0\bmod 3 & s\equiv 1\bmod 3 & s\equiv 2\bmod 3\\
	\hline
	v\equiv 1\bmod 6 & 0 & 1 & 1 \\
	v\equiv 3\bmod 6 & 0 & 0 & -1
	\end{array}
\]
This concludes the proof.
\end{proof}

Now Theorem~\ref{thm:aut2}, which was stated in Section~\ref{sect:introduction}, follows as a combination of Corollary~\ref{cor:ord2} and Lemma~\ref{lma:order2}.

\section{Automorphisms of a binary $q$-analog of the Fano plane}
\label{sect:autfano}
This section is dedicated to the proof of Theorem~\ref{tm:main}.
We prove this assertion by eliminating all other subgroups of $\GL(7,2)$.
This elimination is obtained by a combination of the theoretical results of Section~\ref{sect:triple} and a computer aided search based on the Kramer-Mesner method.

In the following, we shall denote by $\D$ a putative $S_2[2,3,7]$ $q$-Steiner system.
Its automorphism group $\Aut(\D)$ is a subgroup of the group $\GL(7,2) \cong \PGammaL(\F_2^7)$ which has the order
\begin{equation}
\label{eq:ordgl27}
\#\GL(7,2) = 2^{21}\cdot 3^4\cdot 5^1\cdot 7^2\cdot 31^1\cdot 127^1\text{.}
\end{equation}
According to the Sylow theorems, for each prime power $p^r$ with $p^r \mid \#\Aut(\D)$ there exists a subgroup $G \le \GL(7,2)$ such that $\D$ is $G$-invariant.
Hence for each prime factor $p \in \{2,3,5,7,31,127\}$, we strive to find a (preferably small) exponent $r$ such that there is no subgroup $G \le \GL(7,2)$ of order $p^r$ admitting a $G$-invariant $S_2[2,3,7]$ Steiner system.
Then, we can conclude that $\Aut(\D)$ is not divisible by $p^r$.

Since the automorphism groups of isomorphic $q$-Steiner systems are conjugate, it is sufficient to restrict the search to representatives of the subgroups of $\GL(7,2)$ up to conjugacy.
Representatives of the conjugacy classes of the elements are given by the Jordan normal forms (with a fixed order of the Jordan blocks).
This provides an efficient way to create the cyclic subgroups up to conjugacy.%
\footnote{We remark that two non-conjugate elements may generate conjugate subgroups.}
For general subgroup generation, the software package Magma is used \cite{BCP97}.

For all the cases where subgroups were excluded computationally using the Kramer-Mesner method, details are given in Table~\ref{tbl:km}.
The column \enquote{group} contains the group in question.
Its isomorphism type as an abstract group is described in column \enquote{type}.
The columns \enquote{$T$-orb} and \enquote{$K$-orb} contain information on the induced partition of the $2$-subsets and $3$-subsets, respectively.
For example, the entry $4^{644} 2^{42} 1^{7}$ for the group $G_{4,1}$ means that $\gauss{\F_2^7}{2}{2}$ is partitioned into $644$ orbits of length $4$, $42$ orbits of length $2$, and $7$ orbits of length $1$.
If a column of the Kramer-Mesner matrix $\kmmat_{t,k}^G$ contains an entry $> \lambda = 1$, then the corresponding orbit $K^G$ cannot be contained in $\D$.
So we can remove all those columns and the respective variables from the equation system.
The orbit lengths and the size of the reduced system are given in the table columns \enquote{red. $K$-orb} and \enquote{size}.
The reduced system is fed into a solver based on the dancing links algorithm \cite{Knu00}, which is executed on a single core of an Intel Xeon E3-1275 V2 CPU.
The resulting computation time to show the insolvability of the system is given in the last table column \enquote{runtime}.
An entry \enquote{\emph{open}} means that the solver did not terminate within the time limit of a few days.
There are two possibilities where the insolvability can be seen immediately without the need to run the solver:
If the reduced system remains with a zero row (indicated by \enquote{\emph{zero row!}}) or if it is impossible to get the size $\#\D = 381$ as a sum of the lengths of the $K$-orbits (indicated by \enquote{\emph{orbits!}}).

\subsection{Groups of order $2^r$}
From Theorem~\ref{thm:aut2}, we get:

\begin{lemma}\label{lem:2}
If $\D$ is invariant under an automorphism group $G$ of order $2$, then $G$ is conjugate to the cyclic group
\[
	G_2
	= \langle A_{7,3}\rangle
	= \left\langle\left(
	\begin{smallmatrix}
	0& 1& 0& 0& 0& 0& 0 \\
	1& 0& 0& 0& 0& 0& 0 \\
	0& 0& 0& 1& 0& 0& 0 \\
	0& 0& 1& 0& 0& 0& 0 \\
	0& 0& 0& 0& 0& 1& 0 \\
	0& 0& 0& 0& 1& 0& 0 \\
	0& 0& 0& 0& 0& 0& 1 
	\end{smallmatrix}\right)\right\rangle\text{.}
\]
\end{lemma}

\begin{remark}
We attempted to construct a $G_2$-invariant $q$-Steiner triple system $S_2[2,3,7]$ by the Kramer-Mesner method.
The resulting reduced equation system matrix has size $1379 \times 4947$.
Furthermore, from Lemma~\ref{lem:fix_2} we know that $\D$ has exactly $29$ blocks fixed by $G_2$.
However, even with these constraints, the Kramer-Mesner system turned out to be too large to be solved in a reasonable amount of time.
\end{remark}

\begin{lemma}\label{lem:4}
If $\D$ is invariant under a group $G$ of order $4$, then $G$ is a conjugate to the cyclic group
\[
G_{4,1} =
\left\langle
\left(
\begin{smallmatrix}
1&1&0&0&0&0&0 \\
0&1&1&0&0&0&0 \\
0&0&1&0&0&0&0 \\
0&0&0&1&1&0&0 \\
0&0&0&0&1&1&0 \\
0&0&0&0&0&1&1 \\
0&0&0&0&0&0&1
\end{smallmatrix}\right)
\right\rangle
\text{.}
\]
\end{lemma}

\begin{proof} 
Using Magma, we got that there are $42$ subgroup classes of $\GL(7,2)$ of order $4$, falling into $7$ cyclic and $35$ non-cyclic ones.
After removing all groups containing a subgroup of order $2$ of a conjugacy type which is excluded by Lemma~\ref{lem:2}, there remains a single cyclic group $G_{4,1}$ and $7$ non-cyclic groups
\begin{align*}
G_{4,2} &
= \left\langle
A_{7,3},
\left(\begin{smallmatrix}
0 & 1 & 1 & 1 & 0 & 0 & 0 \\
1 & 0 & 1 & 1 & 0 & 0 & 0 \\
0 & 0 & 0 & 1 & 1 & 1 & 0 \\
0 & 0 & 1 & 0 & 1 & 1 & 0 \\
1 & 1 & 1 & 1 & 0 & 1 & 0 \\
1 & 1 & 1 & 1 & 1 & 0 & 0 \\
0 & 0 & 0 & 0 & 0 & 0 & 1
\end{smallmatrix}\right)
\right\rangle\text{,} &
G_{4,3} &
= \left\langle
A_{7,3},
\left(\begin{smallmatrix}
0 & 1 & 0 & 0 & 0 & 0 & 1 \\
1 & 0 & 0 & 0 & 0 & 0 & 1 \\
1 & 1 & 0 & 1 & 1 & 1 & 0 \\
1 & 1 & 1 & 0 & 1 & 1 & 0 \\
1 & 1 & 1 & 1 & 1 & 0 & 1 \\
1 & 1 & 1 & 1 & 0 & 1 & 1 \\
0 & 0 & 0 & 0 & 0 & 0 & 1 
\end{smallmatrix}\right)
\right\rangle\text{,} \\
G_{4,4} &
= \left\langle
A_{7,3},
\left(\begin{smallmatrix}
1 & 0 & 0 & 0 & 1 & 1 & 0 \\
0 & 1 & 0 & 0 & 1 & 1 & 0 \\
0 & 0 & 1 & 0 & 1 & 1 & 0 \\
0 & 0 & 0 & 1 & 1 & 1 & 0 \\
0 & 0 & 1 & 1 & 0 & 1 & 0 \\
0 & 0 & 1 & 1 & 1 & 0 & 0 \\
1 & 1 & 1 & 1 & 0 & 0 & 1
\end{smallmatrix}\right)
\right\rangle\text{,} &
G_{4,5} &
= \left\langle
A_{7,3},
\left(\begin{smallmatrix}
0 & 1 & 1 & 1 & 1 & 1 & 1 \\
1 & 0 & 1 & 1 & 1 & 1 & 1 \\
1 & 1 & 0 & 1 & 0 & 0 & 1 \\
1 & 1 & 1 & 0 & 0 & 0 & 1 \\
1 & 1 & 0 & 0 & 1 & 0 & 1 \\
1 & 1 & 0 & 0 & 0 & 1 & 1 \\
0 & 0 & 0 & 0 & 0 & 0 & 1
\end{smallmatrix}\right)
\right\rangle\text{,} \\
G_{4,6} &
= \left\langle
A_{7,3},
\left(\begin{smallmatrix}
0 & 1 & 1 & 1 & 0 & 0 & 0 \\
1 & 0 & 1 & 1 & 0 & 0 & 0 \\
1 & 1 & 0 & 1 & 1 & 1 & 0 \\
1 & 1 & 1 & 0 & 1 & 1 & 0 \\
1 & 1 & 0 & 0 & 0 & 1 & 0 \\
1 & 1 & 0 & 0 & 1 & 0 & 0 \\
1 & 1 & 0 & 0 & 1 & 1 & 1
\end{smallmatrix}\right)
\right\rangle\text{,} &
G_{4,7} & = 
\left\langle
A_{7,3},
\left(\begin{smallmatrix}
1 & 0 & 0 & 0 & 0 & 0 & 0 \\
0 & 1 & 0 & 0 & 0 & 0 & 0 \\
0 & 0 & 1 & 0 & 0 & 0 & 0 \\
0 & 0 & 0 & 1 & 0 & 0 & 0 \\
1 & 0 & 0 & 1 & 1 & 0 & 0 \\
0 & 1 & 1 & 0 & 0 & 1 & 0 \\
1 & 1 & 0 & 0 & 0 & 0 & 1
\end{smallmatrix}\right)
\right\rangle\text{ and} \\
G_{4,8} & = 
\left\langle
A_{7,3},
\left(\begin{smallmatrix}
0 & 1 & 0 & 0 & 1 & 1 & 1 \\
1 & 0 & 0 & 0 & 1 & 1 & 1 \\
1 & 1 & 1 & 0 & 1 & 1 & 1 \\
1 & 1 & 0 & 1 & 1 & 1 & 1 \\
0 & 1 & 1 & 0 & 1 & 0 & 0 \\
1 & 0 & 0 & 1 & 0 & 1 & 0 \\
1 & 1 & 1 & 1 & 0 & 0 & 1
\end{smallmatrix}\right)
\right\rangle\text{.}
\end{align*}
The latter $7$ groups could be excluded computationally.
\end{proof}

\begin{lemma}\label{lem:8}
The design $\D$ is not invariant under a group of order $8$.
\end{lemma}

\begin{proof}
There are $867$ subgroup classes of $\GL(7,2)$ of order $8$.
After removing all groups containing a subgroup of order $2$ or $4$ of a conjugacy type which is excluded by Lemma~\ref{lem:2} or Lemma~\ref{lem:4}, there remains the single cyclic group
\[
G_{8,1} = \left\langle \left(\begin{smallmatrix}
1&1&0&0&0&0&0 \\
0&1&1&0&0&0&0 \\
0&0&1&1&0&0&0 \\
0&0&0&1&1&0&0 \\
0&0&0&0&1&1&0 \\
0&0&0&0&0&1&1 \\
0&0&0&0&0&0&1
\end{smallmatrix}\right)\right\rangle
\]
and the two quaternion groups
\begin{align*}
G_{8,2} & = \left\langle
\left(\begin{smallmatrix}
1&1&0&0&0&0&0 \\
0&1&1&0&0&0&0 \\
0&0&1&0&0&0&0 \\
0&0&0&1&1&0&0 \\
0&0&0&0&1&1&0 \\
0&0&0&0&0&1&1 \\
0&0&0&0&0&0&1
\end{smallmatrix}\right),
\left(\begin{smallmatrix}
1&0&1&0&1&1&0 \\
0&1&1&0&0&1&0 \\
0&0&1&0&0&0&1 \\
0&1&0&1&1&0&0 \\
0&0&1&0&1&0&0 \\
0&0&0&0&0&1&1 \\
0&0&0&0&0&0&1 \\
\end{smallmatrix}\right)
\right\rangle\text{ and} \\
G_{8,3} & = \left\langle
\left(\begin{smallmatrix}
1&1&0&0&0&0&0 \\
0&1&1&0&0&0&0 \\
0&0&1&0&0&0&0 \\
0&0&0&1&1&0&0 \\
0&0&0&0&1&1&0 \\
0&0&0&0&0&1&1 \\
0&0&0&0&0&0&1 \\
\end{smallmatrix}\right),
\left(\begin{smallmatrix}
1&1&0&0&0&1&0 \\
0&1&0&0&0&0&1 \\
0&0&1&0&0&0&0 \\
1&0&1&1&1&1&0 \\
0&1&1&0&1&0&1 \\
0&0&1&0&0&1&1 \\
0&0&0&0&0&0&1
\end{smallmatrix}\right)
\right\rangle\text{.}
\end{align*}
These $3$ groups are excluded computationally.
\end{proof}

\subsection{Groups of order $3^r$}
\begin{lemma}\label{lem:3}
If $\D$ is invariant under a group $G$ of order $3$, then $G$ is conjugate to one of the cyclic groups
\[
G_{3,1} = \left\langle
\left(
\begin{smallmatrix}
0& 1& 0& 0& 0& 0& 0\\ 
1& 1& 0& 0& 0& 0& 0\\ 
0& 0& 0& 1& 0& 0& 0\\
0& 0& 1& 1& 0& 0& 0\\
0& 0& 0& 0& 0& 1& 0\\
0& 0& 0& 0& 1& 1& 0\\
0& 0& 0& 0& 0& 0& 1
\end{smallmatrix}
\right)
\right\rangle
\quad
\text{and}
\quad
G_{3,2} = \left\langle
\left(
\begin{smallmatrix}
0& 1& 0& 0& 0& 0& 0\\ 
1& 1& 0& 0& 0& 0& 0\\ 
0& 0& 0& 1& 0& 0& 0\\ 
0& 0& 1& 1& 0& 0& 0\\ 
0& 0& 0& 0& 1& 0& 0\\ 
0& 0& 0& 0& 0& 1& 0\\ 
0& 0& 0& 0& 0& 0& 1
\end{smallmatrix}
\right)\right\rangle\text{.}
\]
\end{lemma}

\begin{proof}
There are $3$ conjugacy classes of subgroups of $\GL(7,2)$ of order $3$.
They are represented by $G_{3,1}$, $G_{3,2}$ and
\[
G_{3,3} = \left\langle
\left(
\begin{smallmatrix}
0& 1& 0& 0& 0& 0& 0\\ 
1& 1& 0& 0& 0& 0& 0\\ 
0& 0& 1& 0& 0& 0& 0\\
0& 0& 0& 1& 0& 0& 0\\
0& 0& 0& 0& 1& 0& 0\\
0& 0& 0& 0& 0& 1& 0\\
0& 0& 0& 0& 0& 0& 1
\end{smallmatrix}
\right)\right\rangle\text{.}
\]
The group $G_{3,3}$ could be excluded computationally.
\end{proof}

\begin{lemma}\label{lem:9}
The design $\D$ is not invariant under a group of order $9$.
\end{lemma}

\begin{proof}
There are four conjugacy classes of subgroups of $\GL(7,2)$ of order $9$, a single cyclic and three non-cyclic ones.
After removing all groups containing a subgroup of order $3$ of a conjugacy type which is excluded by Lemma~\ref{lem:3}, there remains the cyclic group
\[
	G_{9,1}
	= \left\langle
	\left(\begin{smallmatrix}
	1 & 0 & 0 & 0 & 0 & 0 & 0 \\
	0 & 0 & 1 & 0 & 0 & 0 & 0 \\
	0 & 0 & 0 & 1 & 0 & 0 & 0 \\
	0 & 0 & 0 & 0 & 1 & 0 & 0 \\
	0 & 0 & 0 & 0 & 0 & 1 & 0 \\
	0 & 0 & 0 & 0 & 0 & 0 & 1 \\
	0 & 1 & 0 & 0 & 1 & 0 & 0
	\end{smallmatrix}\right)
	\right\rangle
\]
and the non-cyclic group
\[
	G_{9,2}
	= \left\langle\left(
	\begin{smallmatrix}
	0 & 1 & 0 & 0 & 0 & 0 & 0 \\
	1 & 1 & 0 & 0 & 0 & 0 & 0 \\
	0 & 0 & 0 & 1 & 0 & 0 & 0 \\
	0 & 0 & 1 & 1 & 0 & 0 & 0 \\
	0 & 0 & 0 & 0 & 0 & 1 & 0 \\
	0 & 0 & 0 & 0 & 1 & 1 & 0 \\
	0 & 0 & 0 & 0 & 0 & 0 & 1
	\end{smallmatrix}\right),
	\left(\begin{smallmatrix}
	1 & 0 & 1 & 1 & 1 & 0 & 0 \\
	0 & 1 & 1 & 0 & 0 & 1 & 0 \\
	0 & 0 & 0 & 0 & 1 & 1 & 0 \\
	0 & 0 & 0 & 0 & 1 & 0 & 0 \\
	0 & 0 & 0 & 1 & 1 & 0 & 0 \\
	0 & 0 & 1 & 1 & 0 & 1 & 0 \\
	0 & 0 & 0 & 0 & 0 & 0 & 1
	\end{smallmatrix}
	\right)\right\rangle\text{.}
\]
These two groups could be excluded computationally.
\end{proof}

\subsection{Groups of other prime power orders}
\begin{lemma}\label{lem:5}
The design $\D$ is not invariant under a group of order $5$.
\end{lemma}

\begin{proof}
As the Sylow $5$-group of $\GL(7,2)$, there is a unique conjugacy class of subgroups of order $5$.
A representative is the cyclic group
\[
G_5 = \left\langle
\left(
\begin{smallmatrix}
0& 0& 0& 1& 0& 0& 0\\
1& 0& 0& 1& 0& 0& 0\\ 
0& 1& 0& 1& 0& 0& 0\\ 
0& 0& 1& 1& 0& 0& 0\\ 
0& 0& 0& 0& 1& 0& 0\\ 
0& 0& 0& 0& 0& 1& 0\\ 
0& 0& 0& 0& 0& 0& 1
\end{smallmatrix}
\right)\right\rangle\text{.}
\]
The action of this group partitions the set of $3$-subspaces into a single fixed point and $2362$ orbits of length $5$.
Because of $\#\D = 381 = 76\cdot 5 + 1$, the fixed $3$-subspace must be contained in $\D$.
After fixing the corresponding variable to the value $1$, the solver returned the insolvability of the system.
\end{proof}

\pagebreak

\begin{lemma}\label{lem:7}
The design $\D$ is not invariant under a group of order $7$.
\end{lemma}

\begin{proof}
The general linear group $\GL(7,2)$ has $3$ conjugacy classes of subgroups of order $7$, represented by
\begin{multline*}
G_{7,1} =
\left\langle\left(\begin{smallmatrix}
0& 1& 0& 0& 0& 0& 0\\
0& 0& 1& 0& 0& 0& 0\\ 
1& 0& 1& 0& 0& 0& 0\\ 
0& 0& 0& 1& 0& 0& 0\\ 
0& 0& 0& 0& 1& 0& 0\\ 
0& 0& 0& 0& 0& 1& 0\\ 
0& 0& 0& 0& 0& 0& 1\\
\end{smallmatrix}\right)\right\rangle
\text{,}\quad
G_{7,2} =
\left\langle\left(\begin{smallmatrix}
0& 1& 0& 0& 0& 0& 0\\
0& 0& 1& 0& 0& 0& 0\\ 
1& 0& 1& 0& 0& 0& 0\\ 
0& 0& 0& 0& 1& 0& 0\\ 
0& 0& 0& 0& 0& 1& 0\\ 
0& 0& 0& 1& 1& 0& 0\\ 
0& 0& 0& 0& 0& 0& 1\\
\end{smallmatrix}\right)\right\rangle \\
\text{and}\quad
G_{7,3} =
\left\langle\left(\begin{smallmatrix}
0& 1& 0& 0& 0& 0& 0\\
0& 0& 1& 0& 0& 0& 0\\ 
1& 0& 1& 0& 0& 0& 0\\ 
0& 0& 0& 0& 1& 0& 0\\ 
0& 0& 0& 0& 0& 1& 0\\ 
0& 0& 0& 1& 0& 1& 0\\ 
0& 0& 0& 0& 0& 0& 1\\
\end{smallmatrix}\right)\right\rangle
\text{.}
\end{multline*}
For the first two groups, the Kramer-Mesner system is immediately seen to be contradictory.
The third system was excluded computationally.
\end{proof}

\begin{lemma}\label{lem:31}
The design $\D$ is not invariant under a group of order $31$.
\end{lemma}

\begin{proof}
As the Sylow $31$-subgroup, $\GL(7,2)$ has a unique conjugacy class of subgroups of order $31$.
A representative is given by the group
\[
G_{31} = \left\langle\left(
\begin{smallmatrix}
1 & 0 & 0 & 0 & 0 & 0 & 0 \\
0 & 1 & 0 & 0 & 0 & 0 & 0 \\
0 & 0 & 0 & 1 & 0 & 0 & 0 \\
0 & 0 & 0 & 0 & 1 & 0 & 0 \\
0 & 0 & 0 & 0 & 0 & 1 & 0 \\
0 & 0 & 0 & 0 & 0 & 0 & 1 \\
0 & 0 & 1 & 0 & 0 & 1 & 0
\end{smallmatrix}
\right)\right\rangle\text{.}
\]
Its action on the set of $3$-subspaces is semiregular, which means that all the orbits are of full length $31$.
Because of $31 \nmid 381 = \#\D$, the design $\D$ cannot arise as a union of a selection of these orbits.
\end{proof}

\begin{lemma}\label{lem:127}
The design $\D$ is not invariant under a group of order $127$.
\end{lemma}

\begin{proof}
As the Sylow $127$-subgroup, $\GL(7,2)$ has a unique conjugacy class of subgroups of order $127$.
Because of $127 = \gauss{7}{1}{2}$, it is the Singer subgroup of $\GL(7,2)$, which has already been excluded in \cite[p.~242]{Tho87}.
\end{proof}

\subsection{Groups of non-prime power order}
The above results on prime-power orders can be combined into the following restriction on the order of $\Aut(\D)$:
\begin{lemma}
\label{lem:1234612}
	\[
		\#\Aut(\D)\in\{1,2,3,4,6,12\}
	\]
\end{lemma}

\begin{proof}
	Assume $\#\Aut(\D) \neq 1$ and let $q$ be a prime power dividing $\#\Aut(\D)$.
	By the Sylow theorems, $\Aut(\D) \leq \GL(7,2)$ contains a subgroup $G$ of order $q$ and $\D$ is $G$-invariant.
	By Equation~\eqref{eq:ordgl27} and Lemma~\ref{lem:8}, \ref{lem:9}, \ref{lem:5}, \ref{lem:7}, \ref{lem:31}, $\ref{lem:127}$ we get $q\in\{2,3,4\}$.
\end{proof}

\pagebreak

\begin{lemma}\label{lem:6}
The design $\D$ is not invariant under a group of order $6$.
\end{lemma}

\begin{proof}
There are $12$ subgroup classes of $\GL(7,2)$ of order $6$, falling into $6$ cyclic ones and $6$ of isomorphism type $S_3$.
After removing all groups containing a subgroup of order $2$ or $3$ of a conjugacy type which is excluded by Lemma~\ref{lem:2} or Lemma~\ref{lem:3}, there remains the single cyclic group
\[
G_{6,1} = \left\langle \left(\begin{smallmatrix}
1 & 0 & 0 & 0 & 0 & 0 & 0 \\
0 & 1 & 1 & 0 & 0 & 0 & 0 \\
0 & 0 & 1 & 0 & 0 & 0 & 0 \\
0 & 0 & 0 & 0 & 1 & 0 & 0 \\
0 & 0 & 0 & 1 & 1 & 1 & 0 \\
0 & 0 & 0 & 0 & 0 & 0 & 1 \\
0 & 0 & 0 & 0 & 0 & 1 & 1
\end{smallmatrix}\right)\right\rangle
\]
and the two groups of isomorphism type $S_3$
\begin{align*}
G_{6,2} & = \left\langle
\left(\begin{smallmatrix}
0 & 1 & 0 & 0 & 0 & 0 & 0 \\
1 & 1 & 0 & 0 & 0 & 0 & 0 \\
0 & 0 & 0 & 1 & 0 & 0 & 0 \\
0 & 0 & 1 & 1 & 0 & 0 & 0 \\
0 & 0 & 0 & 0 & 1 & 0 & 0 \\
0 & 0 & 0 & 0 & 0 & 1 & 0 \\
0 & 0 & 0 & 0 & 0 & 0 & 1 
\end{smallmatrix}\right),
\left(\begin{smallmatrix}
1 & 1 & 0 & 0 & 0 & 0 & 0 \\
0 & 1 & 0 & 0 & 0 & 0 & 0 \\
0 & 0 & 1 & 1 & 0 & 0 & 0 \\
0 & 0 & 0 & 1 & 0 & 0 & 0 \\
0 & 0 & 0 & 0 & 1 & 0 & 0 \\
0 & 0 & 0 & 0 & 1 & 0 & 1 \\
0 & 0 & 0 & 0 & 1 & 1 & 0 
\end{smallmatrix}\right)
\right\rangle\text{ and} \\
G_{6,3} & = \left\langle
\left(\begin{smallmatrix}
0 & 1 & 0 & 0 & 0 & 0 & 0 \\
1 & 1 & 0 & 0 & 0 & 0 & 0 \\
0 & 0 & 0 & 1 & 0 & 0 & 0 \\
0 & 0 & 1 & 1 & 0 & 0 & 0 \\
0 & 0 & 0 & 0 & 0 & 1 & 0 \\
0 & 0 & 0 & 0 & 1 & 1 & 0 \\
0 & 0 & 0 & 0 & 0 & 0 & 1
\end{smallmatrix}\right),
\left(\begin{smallmatrix}
1 & 0 & 0 & 0 & 0 & 0 & 0 \\
1 & 1 & 0 & 0 & 0 & 0 & 0 \\
0 & 0 & 0 & 1 & 0 & 0 & 0 \\
0 & 0 & 1 & 0 & 0 & 0 & 0 \\
0 & 0 & 1 & 0 & 1 & 1 & 0 \\
0 & 0 & 1 & 1 & 0 & 1 & 0 \\
0 & 0 & 0 & 0 & 0 & 0 & 1
\end{smallmatrix}\right)
\right\rangle\text{.}
\end{align*}
These $3$ groups are excluded computationally.
\end{proof}

\begin{lemma}\label{lem:12}
The design $\D$ is not invariant under a group of order $12$.
\end{lemma}
\begin{proof}
By Lemma~\ref{lem:6}, we only need to consider subgroups of $\GL(7,2)$ of order $12$ which do not have a subgroup of order $6$.
The only isomorphism type of a group of order $12$ with this property is the alternating group $A_4$.
However, the Sylow $2$-subgroup of $A_4$ is a Klein four group, which is not possible by Lemma~\ref{lem:4}.
\end{proof}

Now Theorem~\ref{tm:main} follows as a combination of Lemmas~\ref{lem:2}, \ref{lem:4}, \ref{lem:1234612}, \ref{lem:6} and \ref{lem:12}.

\begin{table}[t]
\caption{Kramer-Mesner equation systems}
\label{tbl:km}
\noindent
\noindent\resizebox{\linewidth}{!}{
$\begin{array}{llllllll}
G & \text{type} & T\text{-orb} & K\text{-orb} & \text{red. }K\text{-orb} & \text{size} & \text{runtime} \\
\hline
G_{2} & \Z/2\Z & 2^{1288} 1^{91} & 2^{5800} 1^{211} & 2^{4736} 1^{211} & 1379 \times 4947 & \text{\emph{open}} \\
G_{3,1} & \Z/3\Z & 3^{882} 1^{21} & 3^{3930} 1^{21} & 3^{3720} 1^{21} & 903 \times 3741 & \text{\emph{open}} \\
G_{3,2} & \Z/3\Z & 3^{885} 1^{12} & 3^{3925} 1^{36} & 3^{3710} 1^{36} & 897 \times 3746 & \text{\emph{open}} \\
G_{3,3} & \Z/3\Z & 3^{837} 1^{156} & 3^{3875} 1^{186} & 3^{2170} 1^{186} & 993 \times 2356 & <\text{ 1s} \\
G_{4,1} & \Z/4\Z & 4^{644} 2^{42} 1^{7} & 4^{2900} 2^{98} 1^{15} & 4^{2352} 2^{72} 1^{15} & 693 \times 2439 & \text{\emph{open}} \\
G_{4,2} & (\Z/2\Z)^2 & 4^{616} 2^{84} 1^{35} & 4^{2816} 2^{252} 1^{43} & 4^{2032} 1^{43} & 735 \times 2075 & \text{\emph{zero row!}} \\
G_{4,3} & (\Z/2\Z)^2 & 4^{616} 2^{84} 1^{35} & 4^{2816} 2^{252} 1^{43} & 4^{1792} 1^{43} & 735 \times 1835 & \text{\emph{zero row!}} \\
G_{4,4} & (\Z/2\Z)^2 & 4^{608} 2^{108} 1^{19} & 4^{2824} 2^{228} 1^{59} & 4^{2032} 2^{144} 1^{59} & 735 \times 2235 & <\text{ 1s} \\
G_{4,5} & (\Z/2\Z)^2 & 4^{616} 2^{84} 1^{35} & 4^{2816} 2^{252} 1^{43} & 4^{1840} 1^{43} & 735 \times 1883 & \text{\emph{zero row!}} \\
G_{4,6} & (\Z/2\Z)^2 & 4^{608} 2^{108} 1^{19} & 4^{2816} 2^{252} 1^{43} & 4^{2000} 2^{96} 1^{43} & 735 \times 2139 & \text{\emph{zero row!}} \\
G_{4,7} & (\Z/2\Z)^2 & 4^{608} 2^{108} 1^{19} & 4^{2808} 2^{276} 1^{27} & 4^{1600} 2^{144} 1^{27} & 735 \times 1771 & \text{\emph{zero row!}} \\
G_{4,8} & (\Z/2\Z)^2 & 4^{608} 2^{108} 1^{19} & 4^{2816} 2^{252} 1^{43} & 4^{1632} 2^{192} 1^{43} & 735 \times 1867 & \text{1398m 57s}\\
G_5 & \Z/5\Z & 5^{532} 1^{7} & 5^{2362} 1^{1} & 5^{2107} 1^{1} & 539 \times 2108 & \text{1977m 20s} \\
G_{6,1} & \Z/6\Z & 6^{428} 3^{29} 2^{4} 1^{4} & 6^{1928} 3^{69} 2^{16} 1^{4} & 6^{1464} 3^{54} 2^{14} 1^{4} & 465 \times 1536 & \text{21m 39s} \\
G_{6,2} & S_3 & 6^{400} 3^{85} 2^{3} 1^{6} & 6^{1862} 3^{201} 2^{13} 1^{10} & 6^{998} 3^{156} 2^{7} 1^{10} & 494 \times 1171 & <\text{ 1s} \\
G_{6,3} & S_3 & 6^{399} 3^{84} 2^{7} 1^{7} & 6^{1863} 3^{204} 2^{7} 1^{7} & 6^{904} 3^{162} 2^{7} 1^{7} & 497 \times 1080 & <\text{ 1s} \\
G_{7,1} & \Z/7\Z & 7^{376} 1^{35} & 7^{1685} 1^{16} & 7^{1200} 1^{16} & 411 \times 1216 & \text{\emph{zero row!}}\\
G_{7,2} & \Z/7\Z & 7^{381} & 7^{1687} 1^{2} & 7^{1620} 1^{2} & 381 \times 1622 & \text{\emph{orbits!}} \\
G_{7,3} & \Z/7\Z & 7^{381} & 7^{1686} 1^{9} & 7^{1668} 1^{9} & 381 \times 1677 & \text{2443m 17s} \\
G_{8,1} & \Z/8\Z & 8^{322} 4^{21} 2^{3} 1^{1} & 8^{1450} 4^{49} 2^{7} 1^{1} & 8^{1144} 4^{34} 2^{6} 1^{1} & 347 \times 1185 & \text{29s} \\
G_{8,2} & Q & 8^{322} 4^{19} 2^{6} 1^{3} & 8^{1450} 4^{42} 2^{21} 1^{1} & 8^{1160} 4^{28} 2^{12} 1^{1} & 350 \times 1201 & \text{8m 14s} \\
G_{8,3} & Q & 8^{322} 4^{18} 2^{9} 1^{1} & 8^{1450} 4^{43} 2^{18} 1^{3} & 8^{1160} 4^{16} 2^{18} 1^{3} & 350 \times 1197 & \text{1m 50s} \\
G_{9,1} & \Z/9\Z & 9^{294} 3^{7} & 9^{1310} 3^{7} & 9^{1177} 3^{7} & 301 \times 1184 & \text{57s} \\
G_{9,2} & (\Z/3\Z)^2 & 9^{291} 3^{15} 1^{3} & 9^{1299} 3^{39} 1^{3} & 9^{1077} 3^{27} 1^{3} & 309 \times 1107 & \text{11m 50s} \\
G_{31} & \Z/31\Z & 31^{86} 1^{1} & 31^{381} & 31^{270} & 87 \times 270 & \text{\emph{orbits!}} \\
\end{array}$}
\end{table}

\section{Conclusion}
We have shown that a binary $q$-analog of the Fano plane can only have very few automorphisms.
This result immediately nullifies many natural approaches for the construction, which would inherently imply too much symmetry.

From the point of view of computational complexity though, the vast part of the search space is still untouched, as it consists of the structures without any symmetry.
We believe that further theoretical insight is needed to reduce the complexity to a computationally feasible level.

After all, the question for the existence of a binary $q$-analog of the Fano plane is still wide open.

\section*{Acknowledgements}
The authors are grateful to the organizers of the conference \enquote{Conference on Random Network Codes and Designs over GF$(q)$} held at the Department of Mathematics of Ghent University, Belgium, from September 18--20, 2013, where the first steps towards the present article were initiated.
This conference was organized in the framework of the COST action IC1104, titled \enquote{Random Network Coding and Designs over GF$(q)$}.
Furthermore, we would like to thank the anonymous referees for valuable suggestions improving the readability of the paper.

\end{document}